\documentclass[a4paper,12pt]{amsart}

\usepackage{amsmath}
\usepackage{amssymb,amsbsy,amsmath,amsfonts,amssymb,amscd}
\usepackage{latexsym}
\usepackage{graphics}
\usepackage{color}
\usepackage{longtable}
\usepackage{xy}
\usepackage{array}
\usepackage{color}
\usepackage{comment}
\input xy
\xyoption{all}

\newcommand\sC{{\mathcal C}}
\newcommand\sT{{\mathcal T}}
\newcommand\sD{{\mathcal D}}

\newcommand\sI{{\mathcal I}}

\newcommand\sB{{\mathcal B}}

\newcommand\sX{{\mathcal X}}
\newcommand\sY{{\mathcal Y}}

\newcommand\Lam{\Lambda}
\newcommand\al{\alpha}

\newcommand\Ga{\Gamma}
\newcommand\De{\Delta}
\newcommand\ga{\gamma}

\newcommand{\CC}{\ensuremath{\mathbb{C}}}

\newcommand{\ZZ}{\ensuremath{\mathbb{Z}}}
\newcommand{\QQ}{\ensuremath{\mathbb{Q}}}

\newcommand{\hol}{\ensuremath{\mathcal{O}}}

\newcommand{\PP}{\ensuremath{\mathbb{P}}}

\newcommand{\ra}{\ensuremath{\rightarrow}}

\def\eea{\end{eqnarray*}}
\def\bea{\begin{eqnarray*}}

\newcommand\dual{\mathrel{\raise3pt\hbox{$\underline{\mathrm{\thinspace d
\thinspace}}$}}}
\newcommand\qe{\ifhmode\unskip\nobreak\fi\quad $\Box$}       

\def\BOX{\hfill\lower.5\baselineskip\hbox{$\Box$}}

\newtheorem{theorem}[equation]{Theorem}

\newtheorem{remark}[equation]{Remark}

\newtheorem{defin}[equation]{Definition}

\newtheorem{lemma}[equation]{Lemma}
\newtheorem{example}[equation]{Example}

\newcommand{\sR}{\ensuremath{\mathcal{R}}}

\begin{document}

\title[Orbifolds  Quotients of Tube Domains]{Orbifold Quotients of Symmetric Domains of Tube type}
\author{ Fabrizio  Catanese }
\address{Lehrstuhl Mathematik VIII, Mathematisches Institut der Universit\"{a}t
Bayreuth, NW II, Universit\"{a}tsstr. 30,
95447 Bayreuth,}
 \address{  Korea Institute for Advanced Study, Hoegiro 87, Seoul, 
133-722, South Korea}
\email{Fabrizio.Catanese@uni-bayreuth.de}


\date{\today}

\maketitle

Dedicated to  Lucian Badescu, in memoriam, a dear  friend and  a very esteemed colleague.

\begin{abstract}
In this paper we 
characterize  the compact orbifolds, quotients $ X = \sD/\Ga$
of a bounded symmetric domain $\sD$ of tube type by the  action of a discontinuous  group $\Ga$,
as those projective orbifolds with ample canonical divisor possessing a   slope zero tensor of `orbifold type'.
\end{abstract}

\tableofcontents

\section*{Introduction}

Let $M$ be a simply connected complex manifold, and $\Ga$ be a properly discontinuous group of automorphisms
(biholomorphic self maps) of $M$.

Then the quotient complex analytic space $ X = M / \Ga$ is a normal complex space.

  In the case where the action of $\Ga$ is {\bf quasi-free}, namely, $\Ga$ acts freely outside of a closed complex analytic set
 of codimension at least $2$, we just  consider  the normal complex space $X$; but, in the case where the set $\Sigma$ of points $z \in M$ whose stabilizer is nontrivial has codimension $1$, it is convenient to
 replace $X$ by the complex global orbifold $\sX$, consisting of the datum of $X$ and of the irreducible Weil divisors
 $D_i \subset X$,
whose union is the codimension $1$ part of the branch locus $\sB$ of  $ p : M \ra X$ ($\sB = p (\Sigma)$ is the set 
of critical values of $p$):  each  divisor $D_i$ is marked with the integer $m_i$ which is the order of the stabilizer group at a general point of the inverse image of $D_i$.

The more general case where $\Ga'$ is a properly discontinuous group of automorphisms of  any (connected) complex manifold $M'$ reduces
to the previous one by taking $M$ to be the universal covering of $M'$ and letting $\Ga$ be the group of lifts 
 to $M$ of elements of $\Ga'$.

One says that the above global orbifold $\sX$ is {\bf good} if  $\Ga$ admits a finite index normal subgroup 
$\Lambda$ which acts freely (this holds if  $\Lambda$  is torsion free), with quotient a compact complex manifold $ Y = M / \Lambda$:
in this case $ X = Y/G$, where $G$ is the finite quotient group $G : = \Ga / \Lambda$.

Particularly interesting are the cases where $M$ is a contractible domain  $M = \sD \subset \CC^n$:
in this case $Y$ is a classifying space $K(\Lambda, 1)$ for the group $\Lambda$, whereas (see section 1) $\sX$ is an orbifold classifying
space for the orbifold fundamental group $\Ga$ of $\sX$.

The easiest example is the case where $\sD = \CC^n$, $Y$ is a complex torus $Y = \CC^n / \Lambda$, and $\sX$ is a finite quotient
of a complex torus: this case was considered in \cite{cat24}, and can be  classified by simply saying that $\Ga$
is an arbitrary  abstract even  cristallographic group, endowed with a complex structure on $\Lambda \otimes \CC$.

A more difficult case is the case where $\sD \subset \subset \CC^n$ is a bounded symmetric domain
(see \cite{Koba59}, \cite{Helgason2} and also \cite{MokLibro}, \cite{CaDS}). Again we have a
good global orbifold, by virtue of the so-called Selberg's Lemma (\cite{selberg}, \cite{cassels}). 

 In the case where $\Ga$ acts freely, we have a so-called locally-symmetric manifold, and there
 is
 a vast literature devoted to their possible characterizations, some final touch with rather  explicit 
 criteria being contained in our work with Di Scala, \cite{CaDS}, \cite{CaDS2}.
 
 The purpose of this note is to apply the idea, as in \cite{cat24},  to use orbifolds in order to deal
 with the case of a non  free action of $\Ga$.   At least what is easier here is that necessarily 
 $X$ must be projective, and indeed  by \cite{Kodaira},  the canonical divisor $K_Y$
 and the orbifold canonical divisor $$ K_{\sX} := K_X + \hat{D} := K_X +\sum_i \frac{m_i-1}{m_i}  D_i$$ must be ample.

  We have a priori two  options for the assumptions to be made, for instance we can consider the more
  general orbifolds
  introduced in  \cite{isogenous} (see also \cite{cime}, 5.5 and 5.8, and 6.1 of  \cite{topmethods})
  or the more special Deligne-Mostow orbifolds (\cite{d-m} Section 14), locally modelled as quotients of a smooth manifold by a finite group; in the quasi-free case, where all the multiplicities $m_i = 1$, 
the Orbifold fundamental group $\pi_1^{orb}(\sX)$
 is the fundamental group $\pi_1(X^*)$ of the smooth locus $X^*$ of $X$, while in general $\Ga : = \pi_1^{orb}(\sX)$
is a quotient of $\pi_1(X \setminus D)$.

On the differential geometric side, the input is simpler, see \cite{CaFr},  \cite{CaDS},
 in the case where $\sD$ is of tube type, which means that $\sD$ is biholomorphic to a tube domain 
$\sT = V + i \sC$, where $V$ is a real vector space and $ \sC \subset V$ is an open self-dual convex cone containing no lines.

In fact, the main concept in the tube case  is the one of a nontrivial  {\bf slope zero tensor}
\begin{equation}\label{szt}
 0 \neq \psi_Y \in H^0 ( S^{mn} (\Omega^1_Y) ( - m K_Y )),
 \end{equation}
(here $n : = dim (Y)$) which characterizes the locally symmetric manifolds of tube type, together with the property that $K_Y$
is ample.

Now, $\psi$ descends to a meromorphic section  $0 \neq \psi_X$ on $X$ of 
\begin{equation}
  S^{mn} (\Omega^1_X (log D) )  ( - m (K_X + D )).
 \end{equation}
Conversely,  such a tensor $\psi_X$ on $X$ lifts to a holomorphic tensor $\psi_Y$ only if
it is  {\bf orbifold type}, meaning that some vanishing conditions are to be imposed
(see section 2 for precise definitions).
 
Our  first   result is  the following:

 \begin{theorem}\label{easier}
 The global compact complex orbifolds $\sX$ of bounded symmetric  domains $\sD$
 of tube type (i.e., $\sD$ is a product of irreducible bounded symmetric domains of tube type) 
   are the projective complex orbifolds such that:  
   
   (1) their  orbifold fundamental group $\Ga$ admits a torsion free normal
 finite index subgroup $\Lambda$,
 
  (2) $\sX$ admits a meromorphic slope zero tensor $0 \neq \psi_X$
 ( a meromorphic section of  $S^{mn} (\Omega^1_X (log D) )  ( - m (K_X + D) )$) 
 which is of orbifold type, 

  (3) $ K_{\sX} := K_X +\sum_i \frac{m_i-1}{m_i}  D_i$ is ample,

 and  
 
 (i) the corresponding Galois covering 
 $ Y \ra X= X/G$ ($G : = \Ga / \Lambda$) is smooth, 
 equivalently, the orbifold universal cover of $X$  is smooth, or
 
 (i') $Y$ has singularities which are $2$-homologically connected,
 that is, they have a resolution of singularities $ \pi : Y' \ra Y$ with
 $\sR^j f_* (\ZZ_{Y'})=0, \ {\rm for} \ j=1,2.$

Moreover, $\sX$ should be an  orbifold classifying space.

 \end{theorem}

 Our  main result is instead:
 
 \begin{theorem}\label{main}
 The global compact complex orbifolds $\sX$ of bounded symmetric  domains $\sD$
 of tube type (i.e., $\sD$ is a product of irreducible bounded symmetric domains of tube type) 
   are the projective complex orbifolds such that:  
   
   (1) their  orbifold fundamental group $\Ga$ admits a torsion free normal
 finite index subgroup $\Lambda$,
 
  (2) $\sX$ admits a meromorphic slope zero tensor $0 \neq \psi_X$
 ( a meromorphic section of  $S^{mn} (\Omega^1_X (log D) )  ( - m (K_X + D) )$) 
 of orbifold type, 

  (3)  $ K_{\sX} := K_X +\sum_i \frac{m_i-1}{m_i}  D_i$ is ample,

 and 
  
  (ii) $\sX$ is a  Deligne-Mostow complex projective orbifold, 
   and $\sX$ is an  orbifold classifying space.
 
 \end{theorem}
 
One may speculate/ask whether condition (ii) may be replaced by the weaker assumption that $X$ has KLT singularities.

  \section{Complex orbifolds, Deligne-Mostow orbifolds, orbifold fundamental groups, orbifold coverings}

This section is an abridged version of the corresponding section of \cite{cat24}, so we shall be here
quicker in the exposition.

\begin{defin} (compare 5.5 in \cite{cime}, and section 4 of \cite{d-m})

 Let  $Z$ be a normal complex space, let $D$ be a closed analytic set
of $Z$ containing $Sing(Z)$, and  let $\{ D_i| i \in \sI \}$ be the irreducible components of $D$ of codimension $1$.

(1) Attaching  to each $D_i$ a positive integer $m_i \geq 1$, we obtain a {\bf complex orbifold} $(Z,D, \{ m_i|  i \in \sI \})$.

(2) The {\bf orbifold fundamental group} $\pi_1^{orb} (Z \setminus D, (m_1, \dots, m_r , \dots ))$ is defined as the quotient 
$$ \pi_1^{orb} (Z \setminus D, (m_1, \dots, m_r, \dots ) ) : = \pi_1 (Z \setminus D) / \langle \langle(\ga_1^{m_1}, \dots, \ga_r^{m_r} , \dots \rangle \rangle$$

of the fundamental group of $(Z \setminus D)$ by the subgroup normally generated by simple geometric loops $\ga_i$ going each around a smooth point of the divisor $D_i$ (and counterclockwise). 

(3) The orbifold is said to be {\bf quasi-smooth} or geometric  if moreover $D_i$ is smooth outside of $Sing(Z)$.

(4) The orbifold is said to be a {\bf Deligne-Mostow orbifold} if moreover for each point $z \in Z$
there exists a local chart $\phi : \Omega  \ra U = \Omega/G$, where $0 \in \Omega \subset \CC^n$,
$G$ is a finite subgroup of $GL(n, \CC)$, $\phi(0) = z$, $U$ is an open neighbourhood of $z$, and
the orbifold structure is induced by the quotient map. That is, $D \cap U$ is the branch locus of $\Phi$,
and the integers $m_i$ are the ramification multiplicities. 

(5) An orbifold is said to be {\bf reduced} (or {\bf impure}) if all the multiplicites $m_i =  1$.

(6) We identify two orbifolds under the equivalence relation generated by forgetting the divisors $D_i$ with multiplicity $1$.

\end{defin}

\begin{remark}
(i) A D-M (= Deligne-Mostow) orbifold is quasi-smooth, and the underlying complex space $Z$ has only quotient singularities (these are rational singularities).

(ii) In the case where we have a {\bf reduced} orbifold, that is,  there is no divisorial part,  then the orbifold fundamental group
is the fundamental group of $Z \setminus Sing (Z)$.

(iii) If $Z = M / \Ga$ is the quotient of a complex manifold by a properly discontinuous subgroup $\Ga$, then
$Z$ is a D-M orbifold, since any stabilizer subgroup is finite ($\Ga$ being properly discontinuous) and by Cartan's lemma (\cite{cartan}) the action of the stabilizer subgroup
becomes linear after a local change of coordinates.

(iv) one could  more generally consider a wider class of orbifolds allowing  also the multiplicity $m_i = \infty$: this means that the relation $\ga_i^{m_i}=1$
is a void condition; this more general case is  useful to deal with the compactifications of finite volume quotients $X = \sD / \Ga$ (see for instance \cite{AMRT}).

\end{remark}

Now, to  a subgroup of the orbifold fundamental group corresponds 
a connected  {\bf orbifold covering} of orbifolds,  (see for instance  \cite{d-m}), 
in particular 
to the trivial subgroup corresponds the orbifold universal cover 
$$( \tilde{Z}, \tilde{D}, \{ \tilde{m}_j \}).$$

\begin{defin}
We say that an  orbifold $(Z, D,  (m_j))$ is an orbifold classifying space if its 
universal covering $( \tilde{Z}, \tilde{D}, \{ \tilde{m}_j \})$
satisfies the properties 

(a) either $\tilde{Z}$ is contractible and the multiplicities $\tilde{m}_j $ are all equal to $1$, or

(b) there is a homotopy retraction of $\tilde{Z}$ to a point which preserves the subdivisor $\tilde{D}'$
consisting of the irreducible components with multiplicity $\tilde{m}_j >1$.
\end{defin} 

\begin{defin}
The orbifold canonical divisor is defined as 

 $$ K_{\sX} := K_X + \hat{D} := K_X +\sum_i \frac{m_i-1}{m_i}  D_i.$$
 
 It satisfies the property that, for an orbifold covering $f : \sY \ra \sX$, we have
 $$ f^*(K_{\sX} ) = K_{\sY} .$$

\end{defin} 

\section{Locally symmetric manifolds of tube type and descent of  slope zero tensors}

As mentioned in the introduction,  a bounded symmetric domain 
 $\sD$ is of tube type if  $\sD$ is biholomorphic to a tube domain 
$\sT = V + i \sC$, where $V$ is a real vector space and $ \sC \subset V$ is an open self-dual convex cone containing no lines.

Recall the notation for the classical domains:

\begin{itemize}
\item
   $I_{n,p} $ is the domain $ \sD = \{ Z \in M_{n,p}(\mathbb{C}) :
\mathrm{I}_p - ^tZ \cdot \overline{Z} > 0 \}$.\\
\item
   $II_{n} $ is the intersection of the domain  $I_{n,n} $ with the
subspace of skew symmetric matrices.
   \item
   $III_{n} $ is  the intersection of the domain  $I_{n,n} $
with the subspace of  symmetric matrices.
\item
$IV_{n} $ is the Lie Ball in $\CC^n$,
$$ \{ z | | z_1^2 + \dots + z_n^2| < 1, \ 1 + | z_1^2 + \dots + z_n^2|^2 - 2 (  |z_1|^2 + \dots + |z_n|^2 ) >0$$
\end{itemize}

There are moreover the exceptional domains $\sD_{16}$ of dimension $16$ and 
$\sD_{27}$ of dimension $27$ (related to $2 \times 2$ and $3 \times 3$ matrices over  the Cayley algebra). 

The tube  domain condition excludes the domains $I_{n,p} $ with $ n \neq p$.

The following result was shown in \cite{CaDS}, based on the concept of a slope zero tensor mentioned in
 the introduction, see \eqref{szt}:
$$
 0 \neq \psi_Y \in H^0 ( S^{mn} (\Omega^1_Y) ( - m K_Y )).
$$

\begin{theorem}\label{manifold}

Let $X$ be a compact complex manifold of dimension $n$. Then the
following two conditions:

\begin{itemize}

\item[(1)] $K_X$ is ample
\item[(2)] $X$ admits a nontrivial slope zero   tensor $\psi_X \in
H^0(S^{mn}(\Omega^1_X)(-m K_X) )$, (here $m$ is a positive
integer);

\end{itemize} hold if and only if $X \cong \Omega / \Gamma$ , where
$\Omega$ is a bounded symmetric domain of tube type and $\Gamma$ is a
cocompact
discrete subgroup of
$\mathrm{Aut}(\Omega)$ acting freely.

Moreover, the degrees and the multiplicities of the irreducible
factors of the polynomial $\psi_p$ (the evaluation of $\psi$ at a point $p$) determine uniquely the universal
covering
$\widetilde{X}=\Omega$.

In particular, for $m=2$, we get that  the universal covering
$\widetilde{X}$ is a polydisk if and only if $\psi_p$  is the square of
a squarefree polynomial (indeed, of a product of linear forms).

\end{theorem}

Now, if $X$ is a global orbifold of a symmetric bounded domain $\sD$ of tube type, associated to a properly discontinuous subgroup $\Ga$ of biholomorphisms of $\sD$, then by Selberg's Lemma
the group $\Ga$ admits a finite index normal subgroup $\Lambda$ and the compact complex manifold
$ Y : = \sD / \Lambda$ fulfills  conditions (1) and (2) of Theorem \ref{manifold},
in particular it admits a nontrivial slope zero tensor $\psi_Y$.

Since $ X = Y / G$, $ G = \Ga / \Lambda$,  we want to show that
$\psi_Y$, which is clearly $G$- invariant, descends to some meromorphic tensor $\psi^*_X$ on the smooth locus
$X^*$ of $X$. Then we define 

$$ \psi_X : = i_* ( \psi^*_X), \ \
{\rm for }  \ i : X^* \ra X  \ {\rm being \ the \ inclusion}.$$

In order to achieve our  goal, let us consider now the following local situation.

\begin{lemma}\label{descend}
Consider the action of the cyclic group $\mu_q$ of roots of unity at the origin in $Y : = \CC^n$, via the action,
for $\zeta \in \mu_q$:
$$ (x,z) : = (x_1, \dots, x_{n-1}, z) \mapsto (x, \zeta z), $$
with quotient map $\pi : Y \ra X$,
$$\pi : (x,z) : = (x_1, \dots, x_{n-1}, z) \mapsto (x, y), \ y : = z^q.$$

Then a  $G$-invariant slope zero holomorphic tensor $$ 0 \neq \psi \in H^0 ( S^{mn} (\Omega^1_Y  )  ( - m K_Y  ))$$
descends to a    tensor $ 0 \neq \phi$, a meromorphic section of (here $D = div(y)$)
$$ S^{mn} (\Omega^1_X (log D  ))  ( - m (K_X +   D) ).$$

\end{lemma} 

\begin{proof}
Write $\psi$ as (here $ r  + |I| =  mn$)

$$ \sum_{r=0}^{mn} \sum_{|I| = mn-r} A_{I} (x,z) 
\frac{(dx)^{I} dz^r}{ (dx_1 \wedge \dots \wedge dx_{n-1} \wedge dz)^m} =$$
$$=
\sum_{r=0}^{mn} \sum_{I,h}  B_{I,h} (x) z^h 
\frac{(dx)^{I} dz^r}{ (dx_1 \wedge \dots \wedge dx_{n-1} \wedge dz)^m}=$$
$$ =  \sum_{r=0}^{mn} \sum_{I,h}  B_{I,h} (x)  q^{-r+m} z^{h+r-m}
\frac{(dx)^{I} d(logy)^r}{ (dx_1 \wedge \dots \wedge dx_{n-1} \wedge d(logy))^m}.$$
$G$-invariance is equivalent to the condition that $q$ divides $h+r-m$, hence $h = bq +m -r$,
and we have, downstairs on $X$,  the tensor
$$ \phi : = \sum_{r=0}^{mn}  \sum_{I,h}  q^{-r+m} B_{I,h} (x)  \  y^{b}
\frac{(dx)^{I} d(logy)^r}{ (dx_1 \wedge \dots \wedge dx_{n-1} \wedge d(logy))^m}.$$
The holomorphic part of the tensor $\phi$  is the sum of the series where  $ b\geq 0$, that is,
corresponding to the terms with 
$$ h + r \geq m   .$$
Its order of pole on $D$ is at most $ - b = [\frac{m } {q}]$.

\end{proof}

\begin{defin}
We shall say that the meromorphic tensor $\phi$  is of {\bf orbifold type} if conversely its pull back 
$\psi$ to any orbifold covering with multiplicities $\leq 1$ is holomorphic:
this means that we only have terms with $b$ such that $bq + m  \geq r$.
\end{defin}

\begin{remark}
(i) Since in the case of a good global orbifold quotient the slope tensor of $Y$ cannot vanish on a divisor, the case $h=0$ must occur, hence 
it follows that $q | (m-r)$ if we have a nonzero term with $h=0$ and given $ r$.

(ii) In the case of a polydisk $\sD$, we have semispecial tensors (see \cite{CaDS}), hence we may assume that $m=1$
in the previous Lemma, and it follows that 
 these descend to the quotient as holomorphic tensors.

\end{remark}

 \section{Proof of the Main Theorems}
 
We begin with some general observations, valid also for the speculation made in the introduction.

First of all we concentrate especially on the necessity parts of the statements; observe  that $\sX = Y/G$ is a Deligne-Mostow orbifold and the  singularities of $X$ are quotient singularities,
since, at any point $ y \in Y$ having a nontrivial stabilizer $G_y < G$, the group $G_y$ acts linearly by Cartan's Lemma \cite{cartan}. Since $\sD$ is contractible, $\sX$ is an orbifold classifying space.

By \cite{k-m} (Prop. 5.15, page 158)  quotient singularities $(X,x)$ are rational singularities, that is, they are normal and, if $f : Z \ra X$ is a local resolution, 
then $\sR^if_* \hol_Z=0 $ for $i \geq 1$. They enjoy also the stronger property of being KLT (Kawamata Log Terminal)
singularities.

Indeed Prop. 5.22 of \cite{k-m} (where dlt=KLT if there is no boundary divisor $\De, \De'$) says that KLT singularities are rational singularities, while Prop. 5.20, page 160, says that if we have a finite morphism between normal varieties,
$ F : Y \ra X$, then  $X$ is KLT if and only if $Y$ is KLT).

Hence, conversely, if we start with a Deligne-Mostow  orbifold $X$, the corresponding finite covering $Y$ is Deligne-Mostow,
and if $X$ is KLT, then also $Y$ is KLT.

{\bf First important principle: in both cases (ii), (iii), the normal complex space $Y$ has rational singularities.}

Moreover,  $Y$ is projective if and only if  $X$ is projective (by averaging,  we can find on $Y$ a $G$-invariant very ample divisor).

 We pass now to  the converse implications.
 
 {\bf Key argument: we consider the orbifold covering $Y$ associated to the normal torsion free subgroup 
 $$ \Lam < \Ga : = \pi_1^{orb}(X),$$
 and we show that $Y$ is a locally symmetric manifold.}
 
 \begin{lemma}
 The orbifold $Y$ is just a normal complex space, that is,  there are no marked divisors with  multiplicity $m_i \geq 2$.
 \end{lemma}
 
 \begin{proof}
 Consider the exact sequence
 $$ 1 \ra \Lam \ra \Ga \ra G \ra 1.$$
 
 Then the generators $\ga_i$ have finite order $m_i$, hence their image in $G$ has order exactly $m_i$, because
 $\Lam$ is torsion free.
 
 This means that the covering $ Y \ra X$ is ramified with multiplicity $m_i$ at the divisor $D_i$, and therefore its inverse image
 in $Y$ is a reduced divisor with multiplicity $1$.
 
 \end{proof}

 \subsection{Proof of Theorem \ref{easier}}
 
 Case (i):
 if we assume that $Y$ is smooth, then $Y$ admits a nontrivial slope zero tensor, and we may directly
 invoke Theorem \ref{manifold}, using that by (3) $\sX$ and $Y$ have ample canonical divisor,
 since $K_Y = \pi^* (K_{\sX})$.
 
 Similarly, if the universal covering is smooth, also $Y$ is smooth, because by assumption $\Lambda$
 is torsion free and the stabilizers are finite, whence $\Lambda$ acts freely.
 
 Case (i'): as in \cite{cat24} we prove that $Y$ must be smooth.
 
 Let $Y'$ be a resolution of $Y$. Since by assumption  $\sR^1 f_* (\ZZ_{Y'})=0 $ 
 (this is ture for instance if $Y$ has rational singularities) and 
 $\sR^2 f_* (\ZZ_{Y'})=0, $ we have an isomorphism
 $$ H^j(Y', \QQ) \cong H^j(Y, \QQ), j = 1,2.$$
 
 Hence  the degree $1$ morphism 
 $\al : Y' \ra Y$ yields an isomorphism of first and second cohomology groups.
 
 We follow a similar argument to the one used in \cite{nankai}, proof of Proposition 4.8: it suffices to show 
 that $\al$ is  finite, because then $\al$,  being finite and birational,  is  an isomorphism $ Y' \cong Y$ by normality,
 hence we have shown that $Y$ is smooth and we proceed as for case (i).

  Now, if  $\al$ 
  is not finite,  there is a curve  $C$ which is contracted 
  by $\al$, hence its homology class $c \in H_2(Y', \QQ)$ maps to zero in  $ H_2(Y, \QQ)$.
  And, since  $H^{2}(Y, \QQ) \cong H^{2}(Y', \QQ)$, the class $c$ of $C$, by the projection formula,  is orthogonal  to the whole
  of  $H^{2}(Y', \QQ)$, which is the pull back  of $H^{2}(Y, \QQ)$.

 This is a contradiction because, $Y'$ being projective, 
 the class $c$ of $C$ cannot be trivial.

 \subsection{Proof of Theorem \ref{main}}
 By (ii) 
  $X$ is a   Deligne-Mostow orbifold, hence also $Y$ is a Deligne-Mostow  orbifold, therefore $Y$  is a normal space with  quotient singularities (and these are rational singularities)\footnote{ Similarly, under assumption  (iii), $Y$ has KLT singularities,
which are also rational singularities.}.

  We need now to mimic the proof in the case where $Y$ is smooth, for instance the proof of  Theorem \ref{manifold},
extending it to the case of a normal space $Y$ with  quotient singularities \footnote{Pay attention:  the orbifolds of \cite{campana}
are  the D-M orbifolds  with all $m_i \leq 1$, that we call here of  reduced type!}.

The first ingredient is: the existence of a complete K\"ahler-Einstein metric on the orbifold $Y$ with ample $K_Y$.
 This was first proven in dimension $2$ in \cite{R-Koba} (see also  \cite{Siu87} for the techniques used) and
 was proven  later on in a more general situation in  \cite{casciniLN} and \cite{Eys}.

The second ingredient is  Proposition 5.4 of  \cite{campana}: take the orbifold universal covering $\tilde{Y}$ of $Y$, 
and let $Y'$ be its smooth part: then there exists a De Rham decomposition of Riemannian manifolds
$$ Y' = \prod_i Y'_i,$$
where the holonomy action on each factor is irreducible.

Now, the slope zero tensor is parallel for the Levi Civita connection (by the Bochner principle, since the slope is zero),
as proven by Shoschichi Kobayashi in \cite{Koba80},
hence all the factors have holonomy different from the Unitary group.

Since $K_Y$ is ample, there are no flat factors, and by the Theorem of Berger \cite{Berger} and Simons 
the holonomy of each factor is the holonomy of an irreducible bounded symmetric domain.

Now, the orbifold metric on $\tilde{Y}$ is complete, since the metric on the orbifold $Y$ is complete;
and for each $i$ we can take the completion $\tilde{Y}_i$,
hereby obtaining a decomposition for $\tilde{Y}$.

Now  $Y'_i \subset \tilde{Y}_i$  admits a holomorphic map $f'_i$ to an irreducible  bounded symmetric domain $\sD_i$:
by the Hartogs property, $f_i$ extends to a holomorphic map 
$$f_i : \tilde{Y}_i \ra \sD_i,$$
which is an isometry when restricted to $Y'_i$.

Because of  completeness, $f_i$ is surjective, and since $\tilde{Y}_i $ is normal, its singular locus has codimension $2$, 
hence $ f'_i : Y'_i \ra f'_i (Y'_i)$ must be an isomorphism as the target is simply connected, and we have a covering space.

Again by Hartogs, the inverse of $f_i$, defined on $f'_i (Y'_i)$, extends to yield an isomorphism;
hence $f_i$ is an isomorphism.

In particular, it follows that $Y$ is smooth, hence we only need now to invoke Theorem \ref{manifold}.

\subsection{Final remarks}
    \begin{remark}
  (a)   As already discussed, if $X$ has KLT singularities, by Prop. 5.20 of \cite{k-m} also $Y$ has KLT 
  singularities (these are also  rational singularities).
  In this case, we need again to find a K\"ahler Einstein metric on $Y$, and to use  the Bochner principle.
  
  (b) One may also ask whether one can replace the condition of KLT singularities for $X$ by the condition
  that $X$ has rational singularities, proving  then that also $Y$ has rational singularities.
  \end{remark}

 \bigskip


\end{document}